\documentclass{article}

    \usepackage[colorlinks=true, allcolors=blue]{hyperref}
    \usepackage[margin=1.2in]{geometry}
    \usepackage{charter,authblk,tikz}
    \usepackage[justification=centering]{caption}
    \usepackage[sort,numbers]{natbib}
\usepackage{amsmath,amssymb,amsfonts,amsthm,mathrsfs,dsfont,mathtools,version,enumitem,xcolor,algpseudocode,algorithm,caption,subcaption,multirow,xcolor,textcomp,manyfoot,booktabs,algorithmicx,listings}
\usepackage[nameinlink]{cleveref}

\usepackage[english]{babel}
\usepackage[T1]{fontenc}
\usepackage[title]{appendix}%
\usepackage[xindy, acronyms, nowarn]{glossaries}   
\makeglossaries 
\newcommand{\eg}{{\it e.g.}}
\newcommand{\ie}{{\it i.e.}}
\newglossaryentry{IPG}
{
  name={IPG},
  description={Integer Programming Game},
  first={Integer Programming Game (\glsentrytext{IPG})},
  plural={IPGs},
  descriptionplural={Integer Programming Game},
  firstplural={Integer Programming Games (\glsentryplural{IPG})}
}
\newglossaryentry{PWL}
{
  name={PWL},
  description={piecewise linear},
  first={piecewise linear (\glsentrytext{PWL})},
  plural={PWLs},
  descriptionplural={Piecewise },
  firstplural={piecewise linears (\glsentryplural{PWL})}
}
\newacronym{GCI}{CIG}{Cybersecurity Investment Game} %
\newacronym{MINLP}{MINLP}{Mixed-integer Nonlinear Program} %
\newacronym{MILP}{MILP}{Mixed-integer Linear Program} %
\newacronym{MIQCQP}{MIQCQP}{Mixed-integer Quadratically Constrained Quadratic Program} %
\newacronym{SGM}{SGM}{Sample Generation Method} %

\theoremstyle{plain}
\newtheorem{definition}{Definition}
\newtheorem{thm}{Theorem}

\newtheorem{prop}[thm]{Proposition}

\theoremstyle{definition}

\crefname{lemma}{Lemma}{Lemmata}
\crefname{theorem}{Theorem}{Theorems}
\crefname{claim}{Claim}{Claims}
\crefname{proposition}{Proposition}{Propositions}
\crefname{prop}{Proposition}{Propositions}
\crefname{algorithm}{Algorithm}{Algorithms}
\crefname{equation}{}{}
\crefname{definition}{Definition}{Definition}
\crefname{Cla}{Claim}{Claim}
\crefname{corollary}{Corollary}{Corollaries}
\crefname{remark}{Remark}{Remarks}
\crefname{example}{Example}{Examples}
\crefname{figure}{Figure}{Figures}
\crefname{section}{Section}{Sections}
\crefname{table}{Table}{Tables}
\crefname{enumi}{Statement}{Statements}
\crefname{line}{Step}{Steps}

\newcommand{\fabsapprox}{\text{-absolute approximation }}

\newcommand{\fabsapproxs}{\text{-absolute approximations }}

\newcommand{\theTitle}{Computing Approximate Nash Equilibria for Integer Programming Games}
\newcommand{\theAbstract}{
We propose a framework to compute approximate Nash equilibria in integer programming games with nonlinear payoffs, \ie, simultaneous and non-cooperative games where each player solves a parametrized mixed-integer nonlinear program. 
We prove that using absolute approximations of the players' objective functions and then computing its Nash equilibria is equivalent to computing approximate Nash equilibria where the approximation factor is doubled.
In practice, we propose an algorithm to approximate the players' objective functions via piecewise linear approximations. Our numerical experiments on a cybersecurity investment game show the computational effectiveness of our approach.
}

    \definecolor{lime}{HTML}{A6CE39}
    \DeclareRobustCommand{\orcidicon}{
    	\begin{tikzpicture} \draw[lime, fill=lime] (0,0) circle [radius=0.16] node[white] { {\fontfamily{qag}\selectfont \tiny ID} };
    	\draw[white, fill=white] (-0.0625,0.095) circle [radius=0.007];
    	\end{tikzpicture} \hspace{-2mm}
    }
    
    \foreach \x in {A, ..., Z}{\expandafter\xdef\csname orcid\x\endcsname{\noexpand\href{https://orcid.org/\csname orcidauthor\x\endcsname}{\noexpand\orcidicon} } }

\begin{document}

    \title{\theTitle}
    \author{Aloïs Duguet \orcidA{} Margarida Carvalho \orcidB{} Gabriele Dragotto \orcidC{} \\ Sandra Ulrich Ngueveu \orcidD{}} 
    \date{}

\maketitle
\begin{abstract}\theAbstract\end{abstract}

\section{Introduction} 
\label{section_introduction}

In this paper, we focus on the task of computing Nash equilibria for \glspl{IPG} \citep{Koppe11,IPGs_2023_Tutorial}, a broad class of \emph{simultaneous} and \emph{non-cooperative} games.  In simple terms, an \gls{IPG} is a game among a finite number of players, each of which decides by solving a mixed-integer optimization problem. Compared to more classic game representations, such as normal or extensive form games, \glspl{IPG} can \emph{implicitly} describe the space of strategies (\ie, feasible points) of each player through a mixed-integer programming set. Thus, \glspl{IPG} avoid the possibly expensive \emph{explicit} enumeration of all the players' strategies; this is especially important when the number of strategies available to each player is large or even uncountable, for instance, in a combinatorial setting. We formally define \glspl{IPG} in \cref{def:IPG}.
\begin{definition}[\gls{IPG}]
An \gls{IPG} is a simultaneous and non-cooperative game with complete information among a finite set $M=\{1,2...,m\}$ of players such that each player $p \in M$ solves the parametric optimization problem
\begin{subequations}
\begin{align}
    \max_{x^p} \quad &\Pi^p(x^p;x^{-p}) \\
    \text{s.t.} \quad &x^p \in X^p:= \{A^p x^p \leq b^p, x^p \in \mathbb{R}^{n_p^c} \times \mathbb{Z}^{n_p^i} \},
\end{align}
\label{Player_p_problem}
\end{subequations}
where $x^{-p}=(x^1,\dots,x^{p-1},x^{p+1},\dots,x^m)$ is the vector of strategies for all players except $p$, $X^p$ and $\Pi^p$ are the \emph{strategy set} and the \emph{payoff function} of player $p$, and $A^p$ and $b^p$ are a rational matrix and vector of appropriate dimensions, respectively.
\label{def:IPG}
\end{definition}

When we say an \gls{IPG} is a non-cooperative complete-information game, we mean that each player maximizes its payoff and has full information on the other players' optimization problems, namely, on the objective function and constraints of its opponents.
Similarly to other classes of simultaneous games, the leading solution concept for \glspl{IPG} is the so-called Nash equilibrium. Intuitively, a Nash equilibrium is a \emph{stable} solution where no single player has an incentive to profitably defect from the solution. In recent years, several authors proposed a variety of algorithms to compute Nash equilibria in \glspl{IPG} \citep{Carvalho21,Carvalho22,Cronert22, Dragotto_2021_ZERORegrets,sagratella_computing_2016,schwarze_branch-and-prune_2022}. The majority of these algorithms assume the payoff functions to be linear or linear-quadratic, and often involve the solution of a so-called \emph{best response} program, \ie, the solution of the optimization problem of player $p$~\eqref{Player_p_problem} given a fixed set of other players' strategies $x^{-p}$. As motivated in the tutorial \citep{IPGs_2023_Tutorial}, solving the best response program is pivotal for the correct identification and computation of an equilibrium strategy. However, from a computational perspective, the form of the payoff function $\Pi^p$ intrinsically influences the difficulty of solving the best-response program. Whenever $\Pi^p$ is not linear in $x^p$, the resulting best-response program is a \gls{MINLP}, a well-known class of difficult nonconvex optimization problems \citep{Belotti13}. 
In general, there are different ways to handle nonlinear terms,  \eg, convex relaxations combined with branch-and-bound algorithms \citep{Gounaris08a,Gounaris08b,Liberti04}, or piecewise linear approximations of the nonlinear terms~\citep{Geissler12,Zhang08}. However, to date, the computation of Nash equilibria in \glspl{IPG} with nonlinear utilities (in each player's variables) is a rather unexplored topic, most likely because of the difficulty associated with, on the one hand, computing equilibria, and, on the other hand, handling nonlinearities in a computationally-efficient way.

\paragraph{Contributions. } 
In this paper, we specifically focus on \glspl{IPG} with nonlinear payoffs. We summarize our contributions as follows:
\begin{itemize}
    \item We provide a general methodology to compute Nash equilibria for \glspl{IPG} where players have nonlinear payoff functions in their variables. We prove that performing an approximation with pointwise guarantees on the payoff functions is equivalent to computing an approximate Nash equilibrium.
    \item We propose a piecewise-linear approximation scheme to enable \gls{SGM} from
    \citep{Carvalho22} to compute approximate Nash equilibria in \glspl{IPG} with nonlinear payoffs. Specifically, we approximate the inherently nonlinear players' best-response programs with piecewise-linear approximations.
    \item Finally, we demonstrate the effectiveness of our framework via computational experiments on a cybersecurity investment game. 
\end{itemize}

\paragraph{Outline.} We organize the paper as follows.  \cref{section_background} reviews the current literature and provides some background definition. \cref{section_methodology} presents  our main theorem and the approach to compute approximate Nash equilibria. \cref{section_application} introduces a game-theory model for cybersecurity investments, and the computational results. We propose our conclusions in \cref{section_conclusion}.

\section{Background} \label{section_background}

\paragraph{Approximations.}
Solving \glspl{MINLP} can be a computationally-challenging task, due to the nonlinearities (and nonconvexities) involved in the formulations. A common approach to overcome these challenges is to approximate some of the functions (\eg, the objective function or constraint terms) involved in the formulation, for instance, with an approximation satisfying an absolute error (\cref{def:approx}).
\begin{definition}
    Given a function $f\colon \mathbb X  \longrightarrow\mathbb R$, a function $\hat f\colon \mathbb X  \longrightarrow\mathbb R$ approximates $f$ with an absolute error $\delta>0$ if and only if $| f(x)-\hat f(x)| \leq \delta$ for any $x \in \mathbb X$.
    \label{def:approx}
\end{definition}
The approximation $\hat f$ in \cref{def:approx} guarantees that, for any point in the domain $\mathbb X$ of $f$, the difference between $\hat f$ and $f$ does not exceed $\delta$, \ie, the absolute approximation error. We also denote such $\hat{f}$ a $\delta$\fabsapprox of $f$. In particular, if the function $\hat f$ is a \gls{PWL} function, we say that $\hat f$ is a \gls{PWL} absolute approximation of $f$. The family of \gls{PWL} approximations is rather commonly used, and from a computational standpoint, there exist several methods to construct absolute \gls{PWL} approximations~\citep{Geissler12}. Whenever the original function $f$ has one (real) variable, we can construct a \gls{PWL} absolute approximation by employing the smallest number of pieces~\citep{Rebennack19,Ngueveu19}. However, for functions of two variables or more, there exists no algorithm to compute a \gls{PWL} absolute approximation with the minimal number of pieces \citep{Rebennack15a,Kazda21,Kazda23,Duguet22a}; in addition, the computation times required to build the approximation just in the two-variable case are way larger than in the univariate case.

In the context of \gls{MINLP}, \gls{PWL} approximations are one of the tools to approximate the original problem with a \gls{MILP}. In practice, this means that we can obtain approximate or close to optimal solutions for the original \gls{MINLP} by solving a computationally-easier \gls{MILP}. Naturally, the number of pieces involved in the \gls{PWL} approximation influences the quality of the approximation. On the one hand, a greater number of pieces in the \gls{PWL} approximation results in a tighter approximation. On the other hand, a larger number of pieces may lead to increased computing times, mostly because more pieces require more variables and constraints in the resulting \gls{MILP} \citep{Vielma10}.

\paragraph{Strategies and equilibria.} We denote as $X$ the set of combinations of all players' strategies, \ie, the Cartesian product $X=\prod_{p=1}^m X^p$ of the sets $X^p$ of all players $p$, and we denote any $x \in X$ as a \textit{profile of strategies}. 
For each player $p$, we say $x^p \in X^p$ is a \textit{pure strategy} for $p$. When players randomize over their pure strategies, they play a so-called \textit{mixed strategy}, \ie, a probability distribution $\sigma^p$ over the set of pure strategies $X^p$; let $\Delta^p$ be the space of probability distributions over $X^p$ for player $p$. Similarly to pure strategies, a \textit{profile of mixed strategies} is a vector $\sigma=(\sigma^1,\ldots,\sigma^m)$ of mixed strategies such that $\sigma^p \in \Delta^p$. We call $\Pi^p(\sigma^p;\sigma^{-p})$ the expected payoff of player $p$ associated with $\sigma$.
We employ the \textit{Nash equilibrium} \citep{Nash51} of \cref{def:NE} as a solution concept. 
\begin{definition}
   Given an \gls{IPG} instance and a scalar $\delta \in \mathbb{R}_+$, a profile of mixed strategies $\hat{\sigma}$ is a $\delta$-Nash equilibrium if no player has incentive to deviate from it, i.e., if, for any player $p \in M$, it holds that
   \begin{equation}
       \Pi^p(\hat{\sigma}^i; \hat{\sigma}^{-i}) + \delta \geq \Pi^p(\bar{x}^p;\hat{\sigma}^{-p}) , \quad   \forall \bar{x}^p \in X^p.
       \label{eq:equilibrium}
   \end{equation} 
   Whenever $\delta=0$, we call the Nash equilibrium \emph{exact}; otherwise, if $\delta>0$, we say that the Nash equilibrium is an \emph{approximate} equilibrium.
   \label{def:NE}
\end{definition}

Intuitively, \cref{eq:equilibrium} ensures that, for each player $p$, there exists no unilateral and profitable \emph{deviation} $\bar{x}^p$ such that player $p$ increases its payoff.
From both a computational-complexity and practical perspective, computing exact Nash equilibria is a challenging task; for instance, even in normal-form two-player games~\cite{XiChen2006}, \ie, a class of finite games contained in \glspl{IPG}, computing exact equilibria is far from being trivial. On top of the difficulty of computing an exact equilibrium, even solving a best response program for a player of an \gls{IPG} accounts to solving a \gls{MINLP}, a computationally-challenging problem itself \citep{Belotti13}. 
One way to potentially reduce the computational difficulty of the task of determining an equilibrium is to consider approximate equilibria instead of exact ones. This can, at least from a practical perspective, allow us to approximate the players' best-response programs and compute an approximate equilibrium by solving a series of \gls{MILP} problems.

\section{Our Methodology} 
\label{section_methodology}
In this section, we present our methodology to compute approximate Nash equilibria via \gls{PWL} approximations of the players' payoff functions. Specifically, we link approximate equilibria with approximations having absolute error guarantees. Then, we describe the existent \gls{IPG} and \gls{PWL} we employ, and, finally, we combine our results in an algorithm.

\paragraph{Our assumption.}
Without loss of generality, we represent the payoff of a player $p$ as
$$\Pi^p(x^p;x^{-p}) = f^p(x^p) + g^p(x^p;x^{-p}),$$
where $f^p$ is a function grouping all terms that depend only on the strategy of player $p$ and $g^p$ is a function grouping what we call the \textit{individual terms}. We focus on approximating $f^p$ via \gls{PWL} functions. We remark that we do not approximate $g^p$ since it includes the other players' variables $x^{-p}$ and, therefore, its \gls{PWL} approximation would introduce $x^{-p}$ in the constraints of $p$. In other words, approximating $g^p$ would require the extension of the equilibrium concept to the one of  \emph{generalized Nash equilibrium}, i.e., the equilibrium of a game where the strategy set of each player depends on the strategies of the opponents; in practice, this would prevent us from building on top of the available algorithms for \glspl{IPG}.
We assume that $f^p$ is \emph{a sum of univariate nonlinear functions}. As mentioned in \cref{section_introduction}, this assumption corresponds to the state-of-the-art in terms of approximation methods for \gls{PWL} functions, although any progress in this domain is directly applicable to our methodology.

\subsection{Approximate Equilibria and Payoff Functions}

We present a link between approximate equilibria and approximate payoff functions. Specifically, we prove that, given an \gls{IPG} instance $G$, we can compute an approximate equilibrium by computing an equilibrium of a game $\tilde G$ where we approximate each player's nonlinear payoff with a \gls{PWL} absolute approximation.
\begin{prop} \label{prop}
Let $G$ be an \gls{IPG} where the optimization problem of each player $p$ is of the form~\eqref{Player_p_problem}. Assume that $\hat{G}$ is an \gls{IPG} where each player $p$ solves the optimization problem
$$ \max_{x^p} \{ \hat{\Pi}^p(x^p;x^{-p}) := \hat{f}^p(x^p) + g^p(x^p;x^{-p}) : x^p \in X^p\},$$
where $\hat{f}^p$ is a $\delta$\fabsapprox of $f^p$ in $X^p$.
Then, (i.) a Nash equilibrium $\hat{\sigma}$ of $\hat{G}$ is a $2\delta$-equilibrium of $G$, and 
(ii.) a $\delta_M$-equilibrium of $\hat{G}$ is a $(2\delta+\delta_M)$-equilibrium of $G$.
\end{prop}

\begin{proof}
We first show that (ii.) holds. Consider an arbitrary player $p$ in $G$. Remark that $x^{-p}$ is a parameter in the optimization problem of player $p$, thus the function $\hat{\Pi}^p$ depends only on $x^p$. The function $\hat{\Pi}^p(x^p;x^{-p}) = \hat{f}^p(x^p) + g^p(x^p;x^{-p})$ is thus a $\delta$\fabsapprox of $\Pi^p(x^p;x^{-p}) = f^p(x^p) + g^p(x^p;x^{-p})$ for any $x^{-p}$ because $\hat{f}^p$ is a $\delta$\fabsapprox of $f^p$. %
It follows that
\begin{align}
    \Pi^p(\hat{\sigma}^p;\hat{\sigma}^{-p}) &\geq \hat{\Pi}^p(\hat{\sigma}^p;\hat{\sigma}^{-p}) - \delta, &\label{proof_eq1} \\
    &\geq \hat{\Pi}^p(x^p;\hat{\sigma}^{-p}) - \delta - \delta_M & \quad  \forall x^p \in X^p, \label{proof_eq2} \\
    &\geq \Pi^p(x^p;\hat{\sigma}^{-p}) - 2 \delta - \delta_M & \quad \forall x^p \in X^p, \label{proof_eq3}
\end{align}
where we exploit the fact $\hat{\Pi}$ is a $\delta$\fabsapprox in (\ref{proof_eq1}) and (\ref{proof_eq3}), and the fact $\hat{\sigma}$ is a $\delta_M$-equilibrium in (\ref{proof_eq2}). The above shows that $\hat{\sigma}$ is a $(2 \delta+\delta_M)$-equilibrium of $G$. 
Finally, (i.) holds as a special case of (ii.) by setting $\delta_M = 0$. $\hfill\square$
\end{proof}

Essentially, \cref{prop} claims that a Nash equilibrium of an approximate game $\hat G$ is an approximate equilibrium of the game $G$. As we explain in \cref{subsection_algo}, we employ \cref{prop} to compute $\delta$-equilibria of games where the payoffs are nonlinear.

\subsection{SGM and PWL Approximations} 
\label{subsection_SGM}

In this section, we briefly describe the \gls{SGM} algorithm to compute a Nash equilibrium for \glspl{IPG}~\cite{Carvalho22} and how to build \gls{PWL} approximations. 

\paragraph{The sample generation method.} \gls{SGM} is an iterative method that fundamentally involves two steps. First, \gls{SGM} computes an equilibrium $\sigma$ of an \emph{inner-approximated game}, that is, it computes an equilibrium in a game where the strategy set of each player $p$ is a subset (\ie, an inner approximation) of $X^p$; in practice, we assume that the inner approximation of each player's strategy set is nonempty, \ie, it contains at least one strategy. Second, given the equilibrium $\sigma$ to the inner-approximated game, \gls{SGM} computes the best response of each player $p$ to $\sigma^{-p}$; in other words, it checks, for each player $p$, if there exist unilateral and profitable deviations from $\sigma$. On the one hand, if, for each player $p$ and opponents' strategies $\sigma^{-p}$, the payoff of a best response $\bar{x}^p$ is less than or equal to the one of $\sigma$ plus $\delta$, then  $\sigma$ is a $\delta$-equilibrium for the \gls{IPG}. Equivalently, \gls{SGM} verifies that the stopping criterion $\Pi^p(\sigma^p;\sigma^{-p}) + \delta\ge \Pi^p(\bar{x}^p;\sigma^{-p}) $ holds for each player $p$. On the other hand, if there exists at least one profitable deviation $\bar{x}^p$, \gls{SGM} enlarges the inner approximation of the first step by including $\bar{x}^p$ and iterates again.
When solving the inner-approximated game, \gls{SGM} solves a normal-form game with one of the several algorithms available in the literature. For ease of implementation, and due to the existence of powerful mixed-integer solvers, we employ the so-called \emph{feasibility formulation} from Sandholm et al.~\cite{sandholm2005mixed}\footnote{Although Sandholm et al.~\cite{sandholm2005mixed} concentrates on 2-player games, the formulation can be generalized for $m$-player games in a straightforward manner.}. It also has the advantage of allowing already computed equilibria to be used as warm start. Practically, solving the inner-approximate game accounts to solving a mixed-integer program that is infeasible if no equilibrium exists. Since this is a feasibility program, we also add an objective function that minimizes the support size, \ie, the number of strategies played by each player with positive probability; we employ this objective because there is computational evidence that equilibria tend to have small supports~\cite{Porter2008}.

\paragraph{PWL approximations.}
We approximate the best-response programs in \gls{SGM} with \gls{PWL} $\delta$-absolute approximations of univariate functions. More precisely, when we need to approximate a univariate function $f$ with a \gls{PWL} function $\hat{f}$, we can employ either the algorithm by Codsi et al.~\cite[Algorithm 5]{Codsi21} or the algorithm by Codsi et al.~\cite[Algorithm 4]{Codsi21}. The former algorithm approximates $f$ with the least number of pieces if $f$ is convex or concave \citep[Lemma 4]{Codsi21}, whereas the latter always approximates $f$ with the least number of pieces in the general case, but is slower in practice. As for the algorithm in Codsi et al.~\cite[Algorithm 5]{Codsi21}, it builds a \gls{PWL} function $\hat{f}$ with the least number of pieces such that it is a $\delta$\fabsapprox of $f$. At each iteration, the algorithm constructs a new piece (segment) of $\hat f$ from a specific tangent of the upper-estimator function $f+\delta$ (or lower-estimator $f-\delta$). In \cref{example_lina}, we illustrate the rationale behind this algorithm. Specifically, the red piece intersects the upper estimation $x^3+0.1$ at its endpoints, and intersects the lower estimation $x^3-0.1$ on one point, making it a tangent of the lower estimation. Also, \cref{example_lina_complete} features the \gls{PWL} function created by Codsi et al.~\cite[Algorithm 4]{Codsi21} for the approximation of function $f(x)=x^3$ with absolute error $0.1$ and domain $[-1,1]$.

\begin{figure}[!ht]
\centering
\begin{subfigure}{0.45\textwidth}
    \centering
    \includegraphics[width=\textwidth]{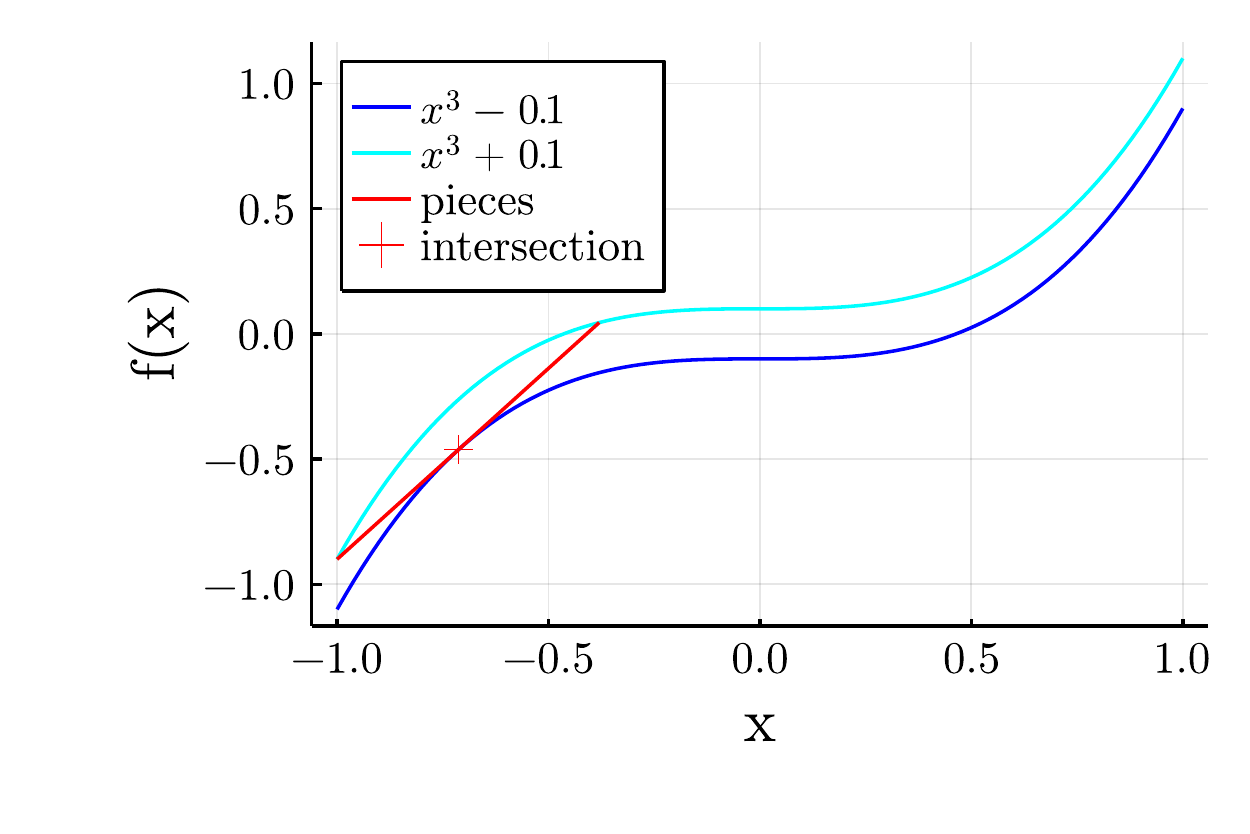}
    \caption{Creation of a piece of the \gls{PWL} approximation by Codsi et al.~\cite[Algorithm 5]{Codsi21}.}
    \label{example_lina}
\end{subfigure}\hfill
\begin{subfigure}{0.45\textwidth}
    \centering
    \includegraphics[width=\textwidth]{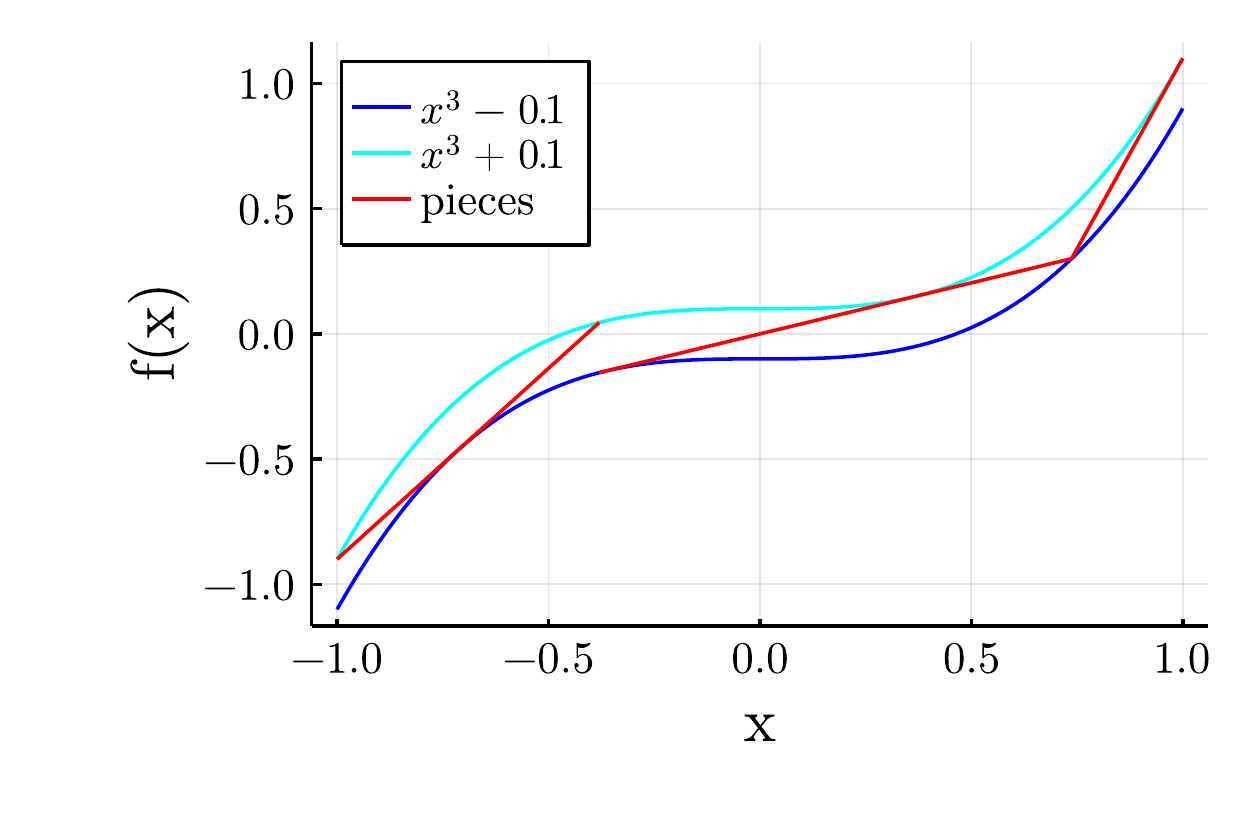}
    \caption{Complete \gls{PWL} approximation with three pieces by Codsi et al.~\cite[Algorithm 4]{Codsi21}.}
    \label{example_lina_complete}
\end{subfigure}
\caption{Illustration of \gls{PWL} approximation.}
\label{example_lina_figure}
\end{figure}

\subsection{Computing Approximate Equilibria} \label{subsection_algo}

In this section, we employ the algorithmic ingredients we introduced to present our main algorithmic procedures for computing approximate equilibria.
Specifically, we propose two algorithmic procedures to compute a $\delta_f$-equilibrium of an \gls{IPG} $G$ based on \gls{SGM} and $\delta_f$\fabsapproxs of the players' payoffs. %

\paragraph{Direct approximation. } In a nutshell, our direct approximation procedure performs a single-round approximation of the players' payoffs, before resorting to \gls{SGM} to compute an equilibrium. In particular, \gls{SGM} will employ the approximated players' payoff functions in the best responses computation. Our direct approximation works as follows. First, for each player $p\in M$ and a parameter $\mu \in ]0,1[$, we approximate the function $f^p$ with a \gls{PWL} $\mu \frac{\delta_f}{2}$\fabsapproxs $\hat{f^p}$. Given $\hat{G}$, \ie, the resulting approximate game of $G$, we compute a $(1-\mu) \delta_f$-equilibrium $\hat{\sigma}$ of $\hat{G}$ with \gls{SGM}.
According to \cref{prop}, $\hat{\sigma}$ is a $2\mu \frac{\delta_f}{2}+(1-\mu) \delta_f=\delta_f$-equilibrium of $G$. 
The parameter $\mu$ determines which proportion of the tolerance $\delta$ is given to the two operations needing a tolerance: the \gls{PWL} approximation and the \gls{SGM}. It thus operates a tradeoff that should be such that, from a computational perspective, the approximation is computationally easy to compute and \gls{SGM} terminates.

\paragraph{$2$-level approximation.} In contrast to the direct-approximation procedure, we introduce a $2$-level approximation procedure where we recursively refine the approximation. The idea is to produce first a $\delta$-Nash equilibrium with $\delta > \delta_f$ so that it is computed faster than the $\delta_f$-Nash equilibrium we are looking for, and then use this approximate equilibrium as a warm start to produce the $\delta_f$-Nash equilibrium. The warm start should reduce drastically the number of iterations inside \gls{SGM} and thus decrease the computation time of the second call to \gls{SGM}. We present our procedure in \cref{algorithm}. Given an integer programming game $G$, a parameter $\mu \in ]0,1[$, a parameter $\delta_f$ and a parameter $\delta_0$ greater than $\mu \frac{\delta_f}{2}$, the algorithm returns a $\delta_f$-equilibrium of $G$. This algorithm uses two procedures. The procedure \texttt{ApproximateIPG}$(G,\delta)$ computes a \gls{PWL} $\delta$\fabsapprox of all the players' payoff functions, creating an approximate game $\hat G$ with the same set of players and strategies, and the payoffs given by the computed $\delta$-approximations. As for the procedure \texttt{SGM}$(G,\delta)$, it finds a $\delta$-equilibrium to game $G$; this procedure can additionally have as input a profile of mixed-strategies functioning as a warm start to the mixed-integer program by~\cite{sandholm2005mixed} used in \gls{SGM} (recall~\cref{subsection_SGM}).

\begin{algorithm}[!ht]
\textbf{Input:} A game $G$, starting tolerance $\delta_0$, final tolerance $\delta_f$ and parameter $\mu \in (0,1)$ \\ %
\textbf{Output:} A $\delta_f$-equilibrium $\hat \sigma$
\begin{algorithmic}[1]
\State $\hat{G}_0 = \texttt{ApproximateIPG}(G,\delta_0)$ \Comment{create the approximate game $\hat{G}_0$}
\State $\hat{\sigma}_0 = \texttt{SGM}(\hat{G}_0,(1-\mu) \delta_f)$ \Comment{find a $(1-\mu) \delta_f$ approximate equilibrium of $\hat{G}$}
\State $\hat{G} = \texttt{ApproximateIPG}(G,\mu \frac{\delta_f}{2})$
\State $\hat{\sigma} = \texttt{SGM}(\hat{G},(1-\mu) \delta_f, \text{warm start} = \hat{\sigma}_0)$ \Comment{use $\hat{\sigma}_0$ as warmstart} 
\State \textbf{return} $\hat{\sigma}$
\end{algorithmic}
\caption{The $2$-level approximation procedure}
\label{algorithm}
\end{algorithm}

\section{Experiments and Computational Results} \label{section_application}
In this section, we test the algorithms introduced in the previous section. 
In the first place, we describe an \gls{IPG} motivated by an existent cybersecurity investment game. Afterward, we detail the experimental setup and we analyze our computational results. In practice, our approximation procedures are efficient and manage to compute approximate equilibria within modest computing times.

\subsection{The Cybersecurity Investment Game} \label{subsection_model}

The \gls{GCI} is a non-cooperative game where a set of players (retailers) sell a homogeneous product in an online marketplace. When selling the product, the transaction can be subject to cyberattacks, which damage the player's reputation and diminish its payoff. This game extends the model of \citep{Nagurney17} by including integer decisions and letting players have nonlinear payoff functions, while still respecting the practical considerations proposed in \citep{Nagurney17}. We refer to \citep{Nagurney17} for a complete description of the motivations behind this model. In this paper, we aim to solve \gls{GCI} to empirically validate the value and performance of the algorithms we introduced in \cref{section_methodology}. 

\subsection{The Model}
In \gls{GCI}, the set of players $M$ represents the retailers selling a homogeneous (or similar) product in a finite set of markets $J=\{1,...,n\}$. Each retailer $p \in M$ decides the quantity $Q_{j}^p$ of product sold in each market $j \in J$ such that $Q_{j}^p$ does not exceed an upper bound $\bar{Q}_{j}^p$. On the one hand, as the transaction takes place online, external attacker can potentially perform a cyberattack on the transaction; if the attack is successful, the retailer is deemed responsible for it, and incurs in financial and reputation damages. On the other hand, each retailer also decides a level of cybersecurity $s^p \in [0,\bar{s}^p]$ to protect its transactions, with $\bar{s}^p < 1$. The higher the level of cybersecurity $s^p$, the least likely cyberattacks will be successful on the transactions of player $p$. Specifically, if $s^p=1$, then no cyberattack on the transactions of player $p$ will succeed. Furthermore, when player $p$ performs a transaction with a market $j$, it incurs in a fixed one-time setup cost, activated through a binary variable $b_{j}^p$, \ie, $b_{j}^p=1$ if and only if player $p$ sells some products in market $j$. All considered, the optimization problem of each player $p$ is the following box-constrained \gls{MINLP}:
\begin{subequations}
\begin{align}
    \label{opt_problem}
        \max \qquad & \sum_{j \in J} \hat{\rho}_{j}(Q^p,s) Q_{j}^p - c^{p}_\text{prod} \sum_{j \in J} Q_{j}^p - \sum_{j \in J} c^{p}_{j,\text{setup}} b_{j}^p  \\
        & - \sum_{j \in J} \left( c^{p}_{j,\text{quad}} (Q_{j}^p)^2 + c^p_{j,\text{lin}} Q_{j}^p \right) - h^p(s^p) - \beta(s) D^p \\
        \text{s.t.} \qquad & Q^p=\sum_{i \in M} \sum_{j \in J} Q^i_{j}, \\
        & 0 \leq Q_{j}^p \leq b_{j}^p \bar{Q}_{j}^p, \quad b_{j}^p \in \{0,1\} \qquad \qquad \forall j \in J, \\
        & 0 \leq s^p \leq \bar{s}^p. \label{opt_problem_c}
\end{align}
\label{Program:Game}
\end{subequations}
The objective function of \cref{opt_problem} contains the following terms:
\begin{enumerate}
    \item \textbf{Profits.} The term $\sum_{j \in J} \hat{\rho}_{j}(Q,s) Q_{j}^p$ represents the profits of player $p$. The variable $Q^p$ is the aggregation of the variables $Q_{j}^p$ for all players $p$ and markets $j$. The function $\hat{\rho}_j(Q^p,s)$ is the unitary selling price fixed by market $j$ given by the inverse demand function $(q_j+ r_j \bar{s}) - m_j \sum_{i \in M} Q_{j}^i $, where $\bar{s}=\frac{1}{m}\sum_{p=1}^m s^p$ is the players' average cybersecurity level, and $q_j$, $m_j$ and $r_j$ are parameters. This price decreases with the total quantity of product sold, and increases with the average cybersecurity level of the players.
    \item \textbf{Production costs.} The term $c^p_{\text{prod}} \sum_{j \in J} Q_{j}^p$ represents the production costs of player $p$, where $c^p_{\text{prod}} \in \mathbb{R}_+$ is a parameter. 
    \item \textbf{One-time setup costs.} The term $\sum_{j \in J} c_{j,\text{setup}}^{p} b_{j}^p$ represents the one-time setup transaction costs player $p$ pays for entering each market $j$, with $c_{j,\text{setup}}^{p}\in \mathbb{R}_+$ being a parameter.
    \item  \textbf{Variable transaction costs.} The term $\sum_{j \in J} \left( c^{p}_{j,\text{quad}} (Q_{j}^p)^2 + c^p_{j,\text{lin}} Q_{j}^p \right)$ represents the variable transaction costs of player $p$, which is a sum of linear and quadratic terms in $Q_{j}^p$. The parameters $c^p_{j,\text{quad}}$ and $c^{p}_{j,\text{lin}}$ are nonnegative real values.
    \item  \textbf{Cybersecurity costs.} The term $h^p(s^p)$ is a function representing the cybersecurity cost of player $p$, bounded by the maximum cybersecurity budget $B^p$. We will analyze this term later in this section.
    \item \textbf{Cybersecurity damage.} The term $\beta(s) D^p $ represents the expected cost player $p$ pays in case of a successful cyberattack. The function $\beta(s)$ is the probability of a successful attack when the security levels are $s=(s^1,...,s^m)$ and $D^p$ is the estimated cost of a successful cybersecurity attack on player $p$. The function $\beta(s)$ is given by $(1-s^p)(1-\bar{s})$.
\end{enumerate}
The difficulty of \cref{Program:Game} comes from the integer variables $b_{j}^p$, the quadratic terms of the players' payoffs and the nonlinear cybersecurity cost $h^p(s^p)$. 

\subsubsection{Cybersecurity Costs}

In contrast to~\citep{Nagurney17}, we upper bound the cybersecurity level $s^p$ to $\bar{s}^p$. Nagurney et al.~\cite{Nagurney17} employ a nonlinear constraint $h^p(s^p)\leq B^p$, with $h^p: [0,1] \mapsto \mathbb{R}^+$ representing the cost of reaching the cybersecurity level $s^p$. Nagurney et al.~\cite{Nagurney17} also assumes $h^p$ is an increasing function, and $h^p(0) = 0$, $h^p(1) = \infty$, because, no matter the amount spent on cybersecurity, in reality, a cyberattack is always a possibility. Under these conditions, $h^p(s^p) \leq B^p \Leftrightarrow s^p \leq \bar{s}^p$ with $\bar{s}^p$ uniquely defined by $h^p(\bar{s}^p) = B^p$; therefore, the nonlinear constraint $h^p(s^p) \leq B^p$ can be equivalently replaced by the linear constraint $s^p \leq \bar{s}^p$, with $\bar{s}^p$ strictly less than $1$ because of the bounds on $h^p$.
Due to these reasons, in our experiments, we use three different functions for the cybersecurity cost: an inverse square root function, a logarithmic function and a nonconvex function. It allows to experiment with best responses of different subclasses of \gls{MINLP}.

\paragraph{Inverse square root function.} The inverse square root function $ h^p_{\text{ISR}}(s^p)$ we consider takes the form of $\textstyle{\alpha^p \left(1/\sqrt{1-s^p} - 1\right)}$ with a positive parameter $\alpha^p$. It is convex on $[0,1[$.
We express $h^p_{\text{ISR}}(s^p)$ as a second-order cone program, with two auxiliary non-negative continuous variables $s_{\mathit{NL}}, t_{\mathit{NL}}$, two quadratic constraints $s_{\mathit{NL}}^2 \leq 1 - s^p$ and $s_{\mathit{NL}} t_{\mathit{NL}} \geq 1$, and a linear term in the objective function $-\alpha^p (t_{\mathit{NL}}-1)$.
Indeed, the term in the objective function forces $t_{\mathit{NL}}$ to be finite because it is a maximization problem. Combined with the constraint $s_{\mathit{NL}} t_{\mathit{NL}} \geq 1$, it forces $s_{\mathit{NL}} > 0$ and $t_{\mathit{NL}} \geq 1/s_{\mathit{NL}}$. In addition, with the constraint $s_{\mathit{NL}}^2 \leq 1 - s^p$, we also have $t_{\mathit{NL}} \geq 1/\sqrt{1-s^p}$. Finally, the objective term $-\alpha^p (t_{\mathit{NL}}-1)$ represents the inverse square root term $\alpha^p/\sqrt{1-s^p} - 1$ because the maximization forces $t_{\mathit{NL}}$ to be as small as possible. With this quadratic formulation of ${\textstyle h^p_{\text{ISR}}}$, the best responses  \cref{Program:Game} are the optimal solutions of a \gls{MIQCQP}.

\paragraph{Logarithmic function.} The logarithmic function we consider is $h^p_{\text{log}}(s^p) = -\alpha^p log(1-s^p)$, which is convex on $[0,1[$.
We can model this function with the linear term $\alpha^p t_{\mathit{NL}}$ in the objective function and the exponential cone constraint $(1-s^p,1,t_{\mathit{NL}}) \in K_{exp}$, where $K_{exp}$ is the convex subset of $\mathbb{R}^3$ defined as 
$\{(x_1,x_2,x_3) : x_1 \geq x_2 \exp(x_3/x_2), x_1,x_2 \geq 0\}$ \citep{mosek}. The best responses \cref{Program:Game} with this formulation of $h^p_{\text{log}}$ are the optimal solution of exponential cone problems.

\paragraph{Nonconvex function.} The nonconvex function we consider is 
$$h^p_{\text{NCF}}(s^p) = \textstyle{\alpha^p (\left(1/\sqrt{1-s^p} + 2/(1+exp(-20 s^p))\right)} - 2.$$
The function $h^p_{\text{NCF}}$ is nonconvex and increasing on $[0,1[$. With this function, the best responses \cref{Program:Game} can be obtained by solving a nonconvex \gls{MINLP}.

\subsection{Numerical Experiments} \label{subsection_numerical}
The presence of integer variables in \gls{GCI} means that the methodology of \citep{Nagurney17} cannot be used as a baseline.
Instead, we employ \gls{SGM} as baseline, and compare it to our approximation procedures: the direct approximation procedure and the $2$-level approximation procedure. 

\paragraph{Instances. } We generate 10 instances for each triplet $(m,n,h^p)$, where $m \in \{2,3,\ldots,7\}$, $n \in \{2,3,\ldots,10\}$ and $h^p = h^p_{\text{ISR}}$, $h^p_{\text{log}}$ or $h^p_{\text{NCF}}$. More precisely, we generate the parameters of those instances with a uniform distribution on regularly sampled intervals close to the values used in the instances of Nagurney et al.~\cite{Nagurney17}, as described in Table \ref{table_intervals}.

\begin{table}[!ht]
\centering
\resizebox{0.7\columnwidth}{!}{
\begin{tabular}{c@{\hspace{10em}}r}
\hline
Parameter & Uniform distribution's domain\\
\hline
   $q_j$ & $\{100,101,...,200\}$\\[0.3em]
   $m_j $ & $ \{0.5,0.51,...,2\}$\\[0.3em]
   $r_j $ & $ \{0.1,0.11,...,0.5\}$\\[0.3em]
   $c^p_{\text{prod}} $ & $ \{1,2,...,10\}$\\[0.3em]
   $c^{p}_{j,\text{setup}}$ & $ \{500,501,...,2000\}$\\[0.3em]
   $c^{p}_{j,\text{lin}}$ & $\{1,1.01,...,4\}$\\[0.3em]
   $c^{p}_{j,\text{quad}} $ &  $\{0.25,0.26,...,1\}$\\[0.3em]
   $\alpha^p $ & $ \{1,2,...,10\}$\\[0.3em]
   $D^p $ & $ \{50,51,...,100\}$\\[0.3em]
    $\bar{Q}_{j}^p$ & $ \{50,51,...,200\}$\\[0.3em]
    $B^p$ & $ \{0.5,1,...,5\}$ \\ \hline
\end{tabular}
}
    \caption{Parameters of the generated instances.}
    \label{table_intervals}
\end{table}

In total, we generate $1620$ instances and we solve them with the \textit{direct approximation procedure}, the \textit{$2$-level approximation procedure}, and with \gls{SGM}. %
We categorize our instances into six subsets to better highlight the factors influencing the performance of the used methods. Specifically, we distinguish: \textit{logarithmic}, \textit{inverse square root} and \textit{nonconvex} cybersecurity cost functions, and \textit{small} ($m = 2,3,4$) and \textit{big} ($m = 5,6,7$) instances. We call the resulting subsets \textit{log234}, \textit{log567}, \textit{root234}, \textit{root567}, \textit{nonconvex234} and \textit{nonconvex567}.

\paragraph{Experimental setup. } The Appendix in Section \ref{appendix} complements the experimental settings described in this paragraph. 

When employing \gls{SGM}, we distinguish between \textit{SGM-MIQCQP}, \textit{SGM-ExpCone} and \textit{SGM-SCIP} depending on the cybersecurity cost function; if we use $h^p_{\text{ISR}}$, the best response solver is Gurobi \cite{gurobi} and the method is SGM-MIQCQP; if it is $h^p_{\text{log}}$, the best response solver is MOSEK \cite{mosek} and the method is SGM-ExpCone; if it is $h^p_{\text{NCF}}$, the best response solver is SCIP \cite{Gamrath20}. When we employ the direct-approximation procedure and the $2$-level approximation procedures, the best response solver is Gurobi. The selection of the above-mentioned solvers is mainly guided by the state of the art of advances in the respective type of best response programs. The exception is SCIP which is a state of the art solver for \glspl{MINLP} among noncommercial solvers\footnote{We do not have access to the commercial state of the art nonconvex solvers according to
the Mittelmann's benchmarks~\cite{Mittelmann2023}.}.

We compute $\delta_f$-approximate equilibria with $\delta_f = 10^{-4}$. To produce a $\delta_f$-equilibrium, we modify the stopping criterion of \gls{SGM} to take into account the error produced by the best response solver. Indeed, a best response solver such as Gurobi or MOSEK uses a relative gap and an absolute gap as stopping criterion. We set the relative gap to $0$, so that only the absolute gap $\delta_{gap}$ remains, i.e., it will return a solution once it is proven not to be more than $\delta_{gap}$ away from the optimum. Thus, the stopping criterion of \gls{SGM} becomes, for a given vector of mixed-strategies $\sigma$,
$$ \max_{x^p \in \hat{X}^p} \hat{\Pi}^p(x;\sigma^{-p}) - \hat{\Pi}^p(\sigma^p;\sigma^{-p}) < \delta_f - \delta_{gap}, \qquad \forall p \in M, $$
where $\hat{\Pi}^p$ and $\hat{X}^p$ are the approximated payoff and strategy set of player $p$. 
We use a \gls{MILP} model to compute the solution of the restricted normal-form game, and we solve it via Gurobi. As opposed to Sandholm et al.~\cite{sandholm2005mixed}, we avoid using explicit Big-$M$ formulations, and we employ so-called indicator constraints \citep{bonami2015mathematical}; in practice, this will let the solver decide the best strategy to reformulate the logical conditions. 
For Algorithm \ref{algorithm} of the $2$-level approximation procedure, we set the parameter $\mu$ to $0.5$ and the initial approximation level $\delta_0$ to $0.05$. The value of $\delta_0$ is chosen as a balance between relatively small approximation error and relatively small best response models; a tolerance of $\delta_0=0.05$ implies that a $\delta$-Nash equilibrium computed with it would be at least a $\delta$-Nash equilibrium with $\delta = 0.1$, which is $1000$ times $\delta_f$ used in the experiments. Our goal with this value is to show that using a really large value $\delta_0$ compared to $\delta_f$ can still be sufficient to find a $\delta_f$-Nash equilibrium faster than for the direct approximation.
We set a time limit of $900$ seconds for \gls{SGM}, as the time spent outside of \gls{SGM} is often small (\eg, less than $5$ seconds). We employ an Intel Core i5-10310U CPU, with $32$ GB of RAM, and Gurobi~ $9.5$ and MOSEK~$10.0$ running on $4$ threads and SCIP 7.0.3.

\begin{figure}[!ht]
\centering
\begin{subfigure}{.5\textwidth}
  \centering
 \includegraphics[width=0.9\textwidth]{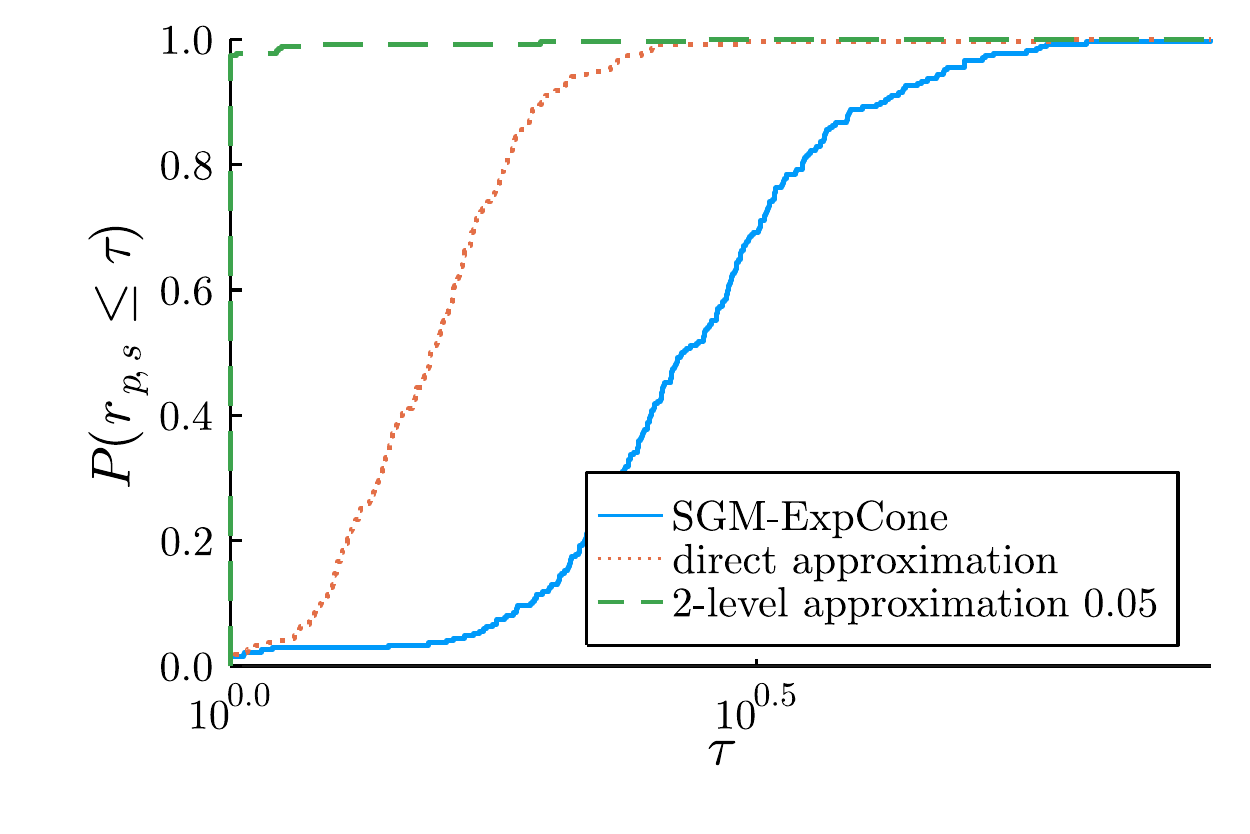}
    \caption{Results with 2 to 4 players (log234 instances).}
    \label{figure_log234}
\end{subfigure}%
\begin{subfigure}{.5\textwidth}
  \centering
  \includegraphics[width=0.9\textwidth]{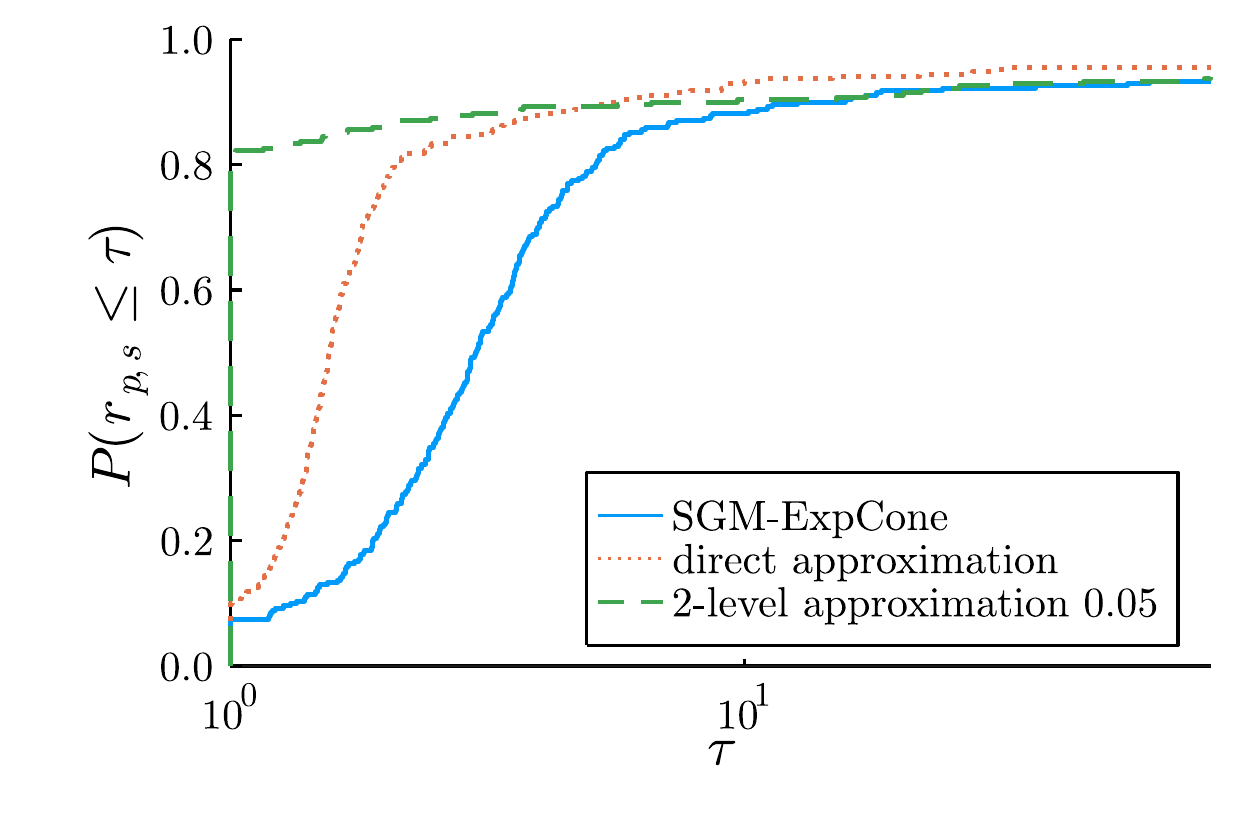}
    \caption{Results with 5 to 7 players (log567 instances).}
    \label{figure_log567}
\end{subfigure}
\caption{Performance profiles with logarithmic cybersecurity cost function.}
\label{figure_log}
\end{figure}

\begin{figure}[!ht]
\centering
\begin{subfigure}{.5\textwidth}
  \centering
 \includegraphics[width=0.9\textwidth]{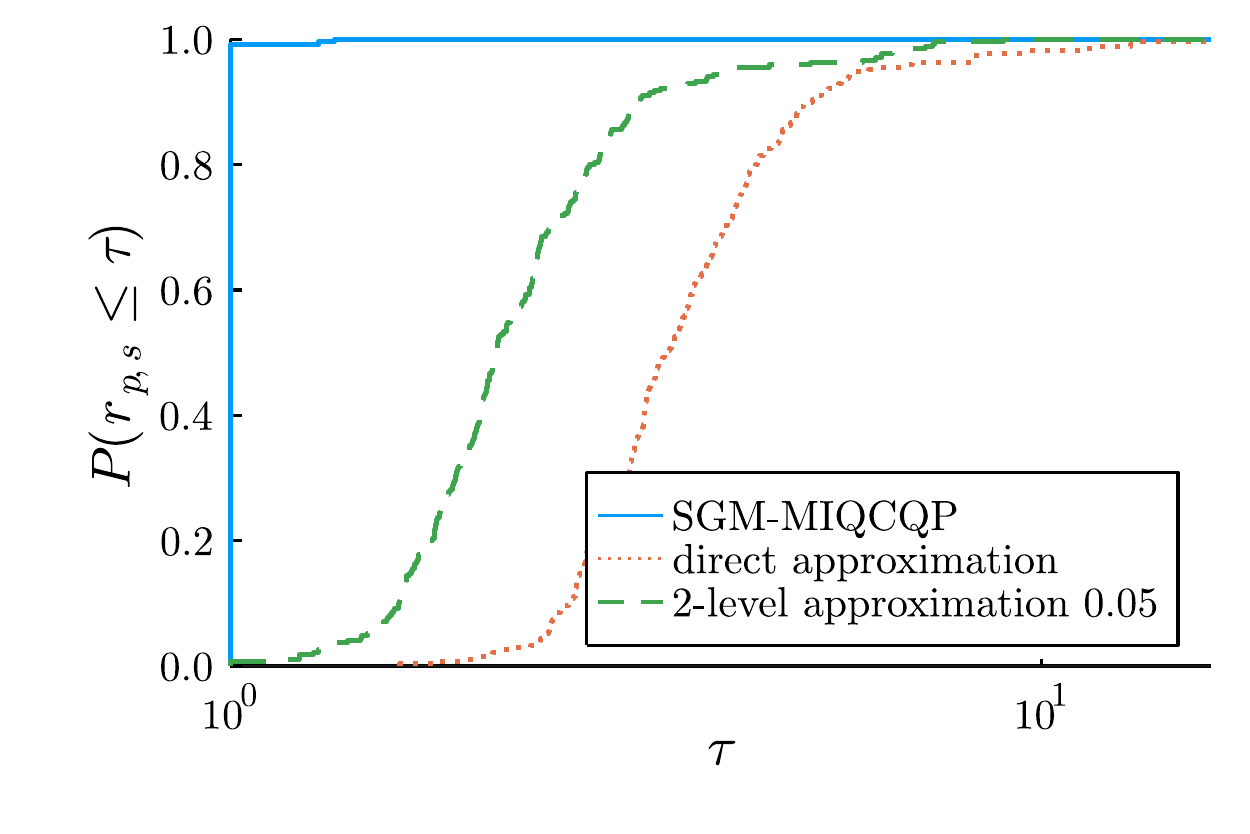}
    \caption{Results with 2 to 4 players (root234 instances).}
    \label{figure_root234}
\end{subfigure}%
\begin{subfigure}{.5\textwidth}
  \centering
  \includegraphics[width=0.9\textwidth]{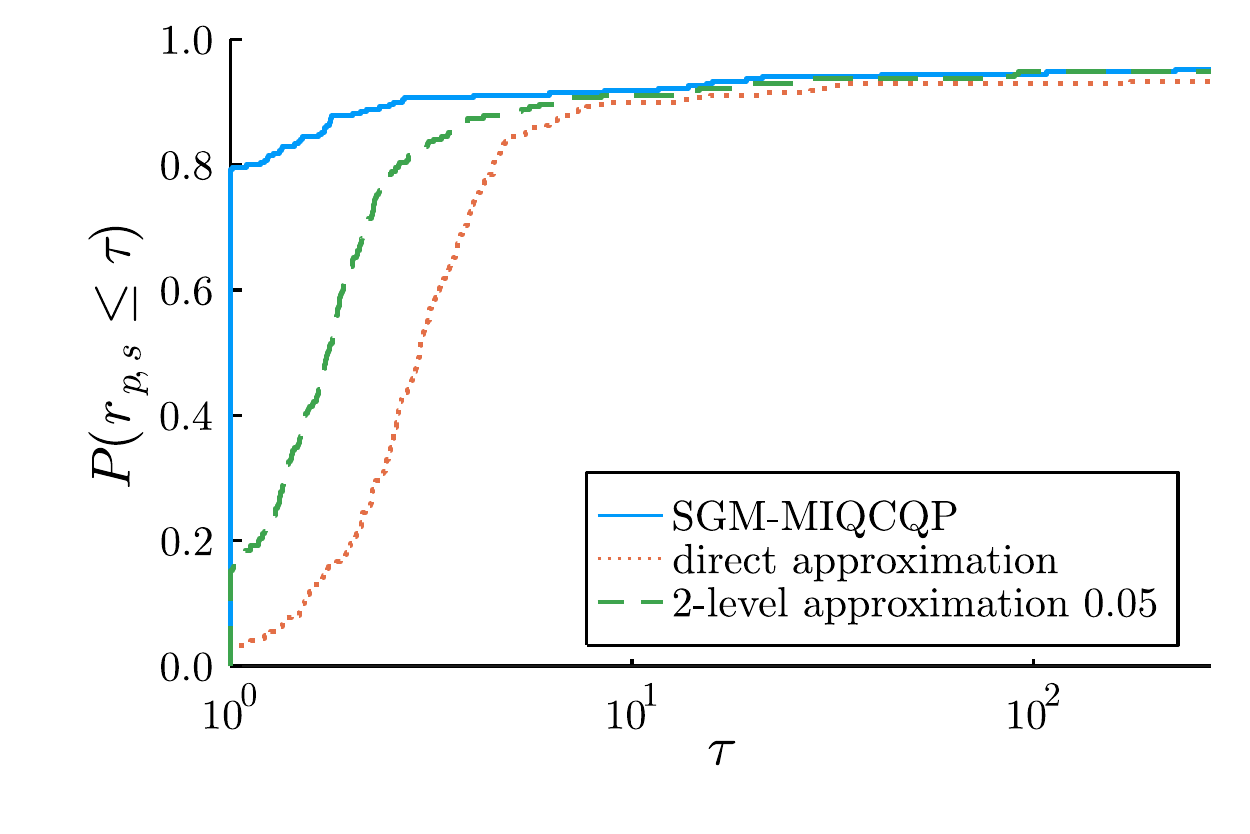}
    \caption{Results with 5 to 7 players (root567 instances).}
    \label{figure_root567}
\end{subfigure}
\caption{Performance profiles with inverse square root cybersecurity cost function.}
\label{figure_root}
\end{figure}

\begin{figure}[!ht]
\centering
\begin{subfigure}{.5\textwidth}
  \centering
 \includegraphics[width=0.9\textwidth]{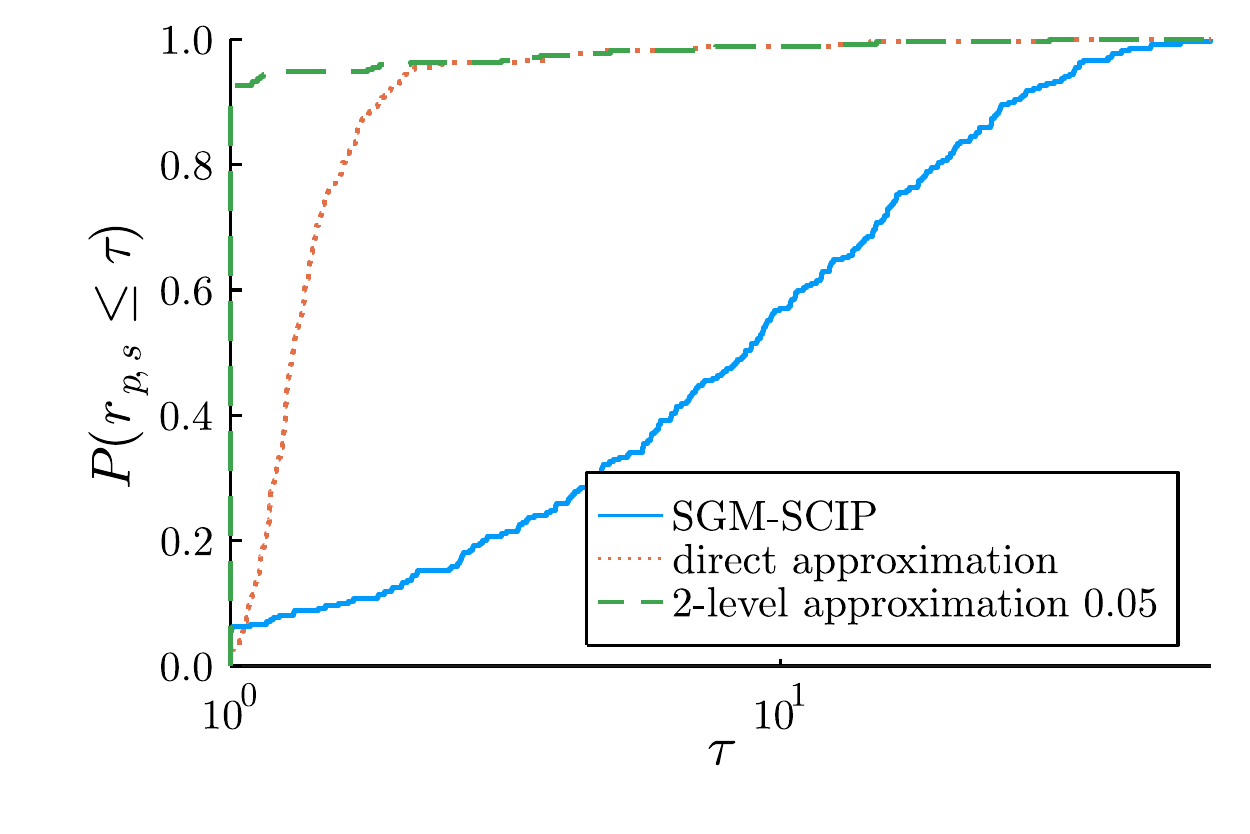}
    \caption{Results with 2 to 4 players (nonconvex234 instances).}
    \label{figure_nonconvex234}
\end{subfigure}%
\begin{subfigure}{.5\textwidth}
  \centering
  \includegraphics[width=0.9\textwidth]{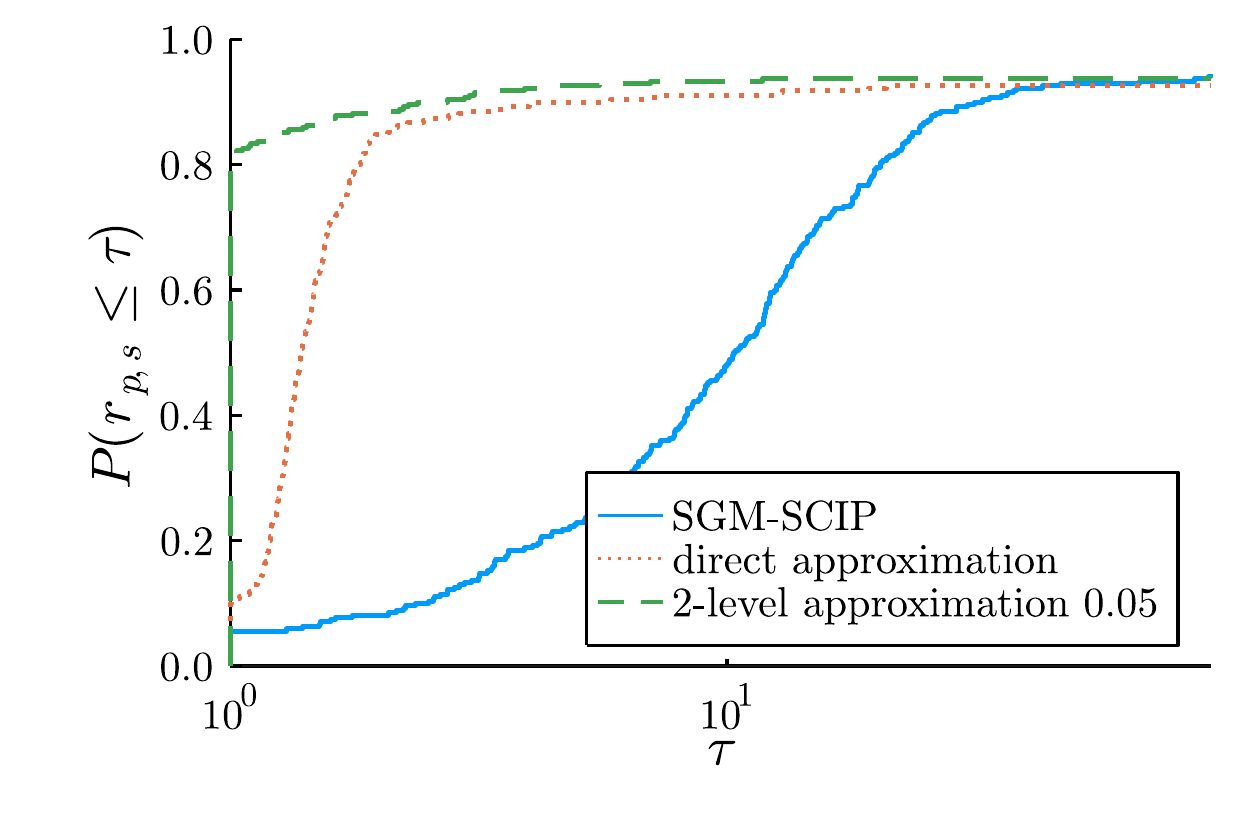}
    \caption{Results with 5 to 7 players (nonconvex567 instances).}
    \label{figure_nonconvex567}
\end{subfigure}
\caption{Performance profiles with nonconvex cybersecurity cost function.}
\label{figure_nonconvex}
\end{figure}

We present the experimental results as performance profiles in \cref{figure_log,figure_root}. According to \citep{dolan02}, we call performance ratio of a given instance with solver $s$ and instance $i$ among a set of solvers $S$ the number greater or equal to $1$ computed as $t_{i,s} / \min_{s' \in S} t_{i,s'}$, where $t_{i,s}$ is the computation time of solver $s$ for instance $i$. The performance profiles show the proportion of instances solved by the corresponding solver with a performance ratio of no more than $\tau$. Finally, we display the performance ratio with a logarithmic scale. Table \ref{table_statistics} shows the percentage of instances solved, the geometric mean time and the average number of iterations in \gls{SGM} grouped by the method and the subset of instances. The geometric mean is computed as $M = \sqrt[n]{\Pi_{i \in 1,...,n} t_i}$ with $t_i$ the computation time of instance $i$.  We use it instead of the arithmetic mean because most of the instances are solved in less than $10$ seconds, and, therefore, the arithmetic mean is greatly influenced by the few instances solved with a time close to the time limit of $900$ seconds.

\begin{table}[!ht]
\centering
\resizebox{\textwidth}{!}{
\begin{tabular}{ll@{\hspace{2em}}rrr@{\hspace{2em}}rrr@{\hspace{2em}}rrr}\hline 
        & Type & \multicolumn{3}{c}{\textbf{SGM}} & \multicolumn{3}{c}{\textbf{direct approximation}} & \multicolumn{3}{c}{\textbf{$2$-level approximation}} \\
        && \% solved & time (s) & iter. & \% solved & time (s) & iter. & \% solved & time (s) & iter. \\ \hline %
\textbf{log234} & Exp. cone & 100 & 1.04 & 15.61 & 100 & 0.59 & 16.29 & 100 & 0.38 & 16.32 \\
\textbf{log567} & Exp. cone & 93.3 & 5.06 & 38.44 & 95.6 & 3.22 & 41.5 & 94.1 & 2.19 & 41.17 \\\hline %
\textbf{root234} & MIQCQP & 100 & 0.21 & 15.80 & 100 & 0.76 & 16.19 & 100 & 0.47 & 16.23 \\
\textbf{root567}& MIQCQP & 95.2 & 1.63 & 39.12 & 93.7 & 3.87 & 39.98 & 94.8 & 2.55 & 40.53 \\\hline
\textbf{nonconvex234} & MINLP & 100 & 12.06 & 15.83 & 100 & 2.36 & 16.23 & 100 & 1.8 & 16.14 \\
\textbf{nonconvex567} & MINLP & 94.4 & 47.54 & 41.38 & 92.6 & 7.68 & 38.64 & 93.7 & 5.68 & 38.62 \\\hline
\end{tabular}
}
\caption{Overview of computational results.}
\label{table_statistics}
\end{table}

\paragraph{Analysis. } %
All of the smaller instances (\ie, $m\le 4$) are solved within the time limit, while $5.9\%$ of the larger instances (\ie, $m \ge 5$) are not solved within it.
Among larger instances, we observe that the number of time limit hits per method is similar, and varies between $4.8\%$ and $7.4\%$.

There is a relative difference in the performance of the methods depending on the nonlinear cost function used.
The instances in log234, log567, nonconvex234 and nonconvex567 are solved more efficiently with the $2$-level approximation procedure, with more than $80\%$ of the instances solved in less time than for the other methods.
Also, the performance ratios of the $2$-level approximation $0.05$ procedure are strictly the best for $\tau$ below $3$ for both log234 and log567 instances. 
However, instances root234 and root567 are solved faster using SGM-MIQCQP with respectively all instances and $78\%$ of them solved in less time than for the other methods. 
Moreover, the performance ratios of SGM-MIQCQP are strictly the best for $\tau$ below $20$ and $10$, respectively. 
Regarding instances nonconvex234 and nonconvex567, the performance ratios of the method $2$-level approximation $0.05$ are the best for $\tau$ below $2$.
Table \ref{table_statistics} shows the average number of iterations in \gls{SGM}, depending on the method and the subset of instances. 
Regarding the number of iterations for the $2$-level approximation procedure in Table \ref{table_statistics}, it shows the number of iterations of the first run of \gls{SGM}, because the second iteration systematically takes less than $5$ iterations. This is explained by the fact that the second iteration starts quite close to an approximate Nash equilibrium thanks to the warm start. The average number of iterations does not depend much on the method, highlighting that the differences in computation times are mostly due to the computation of best responses and normal-form equilibria performed in \gls{SGM}.

To conclude this analysis, the \gls{PWL} approximation-based methods give better results than a direct application of \gls{SGM} on instances with the logarithmic function and the nonconvex function. However, the results are reversed on instances with the inverse square root function. Those results may be explained by observing the initial class of the optimization problem that models the best response : when calculating approximate Nash equilibria in IPGs with nonlinear payoff functions, the more difficult it is to solve the best response problem, the more efficient the methods based on the \gls{PWL} approximation.

\section{Conclusions} \label{section_conclusion}

In this work, we proposed a framework to compute approximate Nash equilibria in \glspl{IPG} with nonlinear payoff functions. We proved that approximating the payoff functions is equivalent to computing approximate equilibria. From a practical standpoint, we introduced two procedures based on \gls{PWL} approximations to compute approximate Nash equilibria. Finally, we proved the efficiency of our method by proposing and solving a cybersecurity investment game, establishing a computational benchmark for our methods. Our method seems best indicated when the type of optimization problems encountered in the best responses is not efficiently solved by state of the art solvers.
We believe there are several directions to extend our work. Among those, we propose three ideas.  
The first idea is the extension of \cref{prop} for different types of pointwise approximation errors, \eg, relative errors \citep{Ngueveu19}. 
Second, we believe the nonlinear functions can be approximated by \gls{PWL} approximations refined only on some well-chosen intervals or in an iterative fashion so that those approximations have less pieces.
Third, in addition to \glspl{IPG}, our methods can employ different algorithms to compute equilibria, \eg, the Cut-and-Play algorithm from~\citep{Carvalho21}. %

\section*{Acknoledgement}
The authors are thankful for the support of the ETI program of Toulouse INP through the project POLYTOPT, and Institut de valorisation des donn\'ees (IVADO) and Fonds de recherche du Qu\'ebec (FRQ) through the FRQ-IVADO	Research Chair and NSERC grant 2019-04557. Gabriele Dragotto is thankful for the support of the \emph{Data X} program from the \emph{Center for Statistics and Machine Learning} at Princeton University.

\section*{Data and Code Availability}
The code, instances and the results of the experiments are available at \url{https://github.com/LICO-labs/SGM-and-PWL}.

\section*{Appendix} \label{appendix}
\subsection*{Implementation Details}
We describe here some technical details on the experimental part, such as values for some parameters of solvers used.

The formulation of the \gls{PWL} functions approximating the convex functions (inverse square root and logarithmic functions) are decided by the use of the Gurobi function \textit{addGenConstrPWL}. It is not the case for the nonconvex function because LinA computes a continuous \gls{PWL} function in the case of the approximation of a convex function and \textit{addGenConstrPWL} can only represent continuous \gls{PWL} functions. Gurobi's \textit{addGenConstrPWL} is used because it is tailored to handle a \gls{PWL} relationship $y = f(x)$. The formulation used for the \gls{PWL} approximation of the nonconvex function is the disaggregated logarithmic convex combination model \citep{Vielma10}.

We set the absolute gap $\delta_{gap}$ allowed for the best response solver and the \gls{MILP} model of the normal-form game to $(1-\mu) \frac{4\delta_f}{5}$. Also, for numerical reasons, a value of $10^{-9}$ is affected to the parameters \textit{IntFeasTol} and \textit{FeasibilityTol} of Gurobi, \textit{mioTolAbsRelaxInt} and \textit{mioTolFeas} of MOSEK and \textit{numerics/feastol} of SCIP. It sets the feasibility tolerance and the tolerance to satisfy integrality constraints for integer variables to a common value for all best response solvers.

To approximate the cybersecurity cost functions with \gls{PWL} functions, we use the Julia package \texttt{LinA} \citep{Codsi21}, an open-source implementation of \citep[Algorithm 5]{Codsi21} available at %
\url{https://github.com/LICO-labs/LinA.jl}. As the inverse square root function and the logarithmic function are convex, this algorithm produces a \gls{PWL} approximation with the minimum number of pieces satisfying the absolute approximation error. Regarding the nonconvex function, the algorithm produces a \gls{PWL} approximation with at most one more piece than the minimum according to \citep[Lemma 4]{Codsi21} and the fact that the nonconvex function is convex and then concave on $[0,1]$. \citep[Algorithm 4]{Codsi21} produces a \gls{PWL} approximation with the minimum number of pieces but it takes a higher computation time. It was not used to approximate the nonconvex function because we estimated it was not worth to spend more time computing this approximation to get at best one piece less. 

The best response models solved with Gurobi are modeled with the gurobipy library, while the best response models solved with MOSEK are modeled with the pyomo library version 6.4. 
The main part of the code concerning the \gls{PWL} approximations is in Julia~1.6, while \gls{SGM} is coded in Python~3.8. Thus there is a little delay each time \gls{SGM} is called to load the Python environment and libraries. 
This Python loading time is removed from the computation time because it could have been removed by an implementation in a single language.

\bibliography{biblio.bib}
\end{document}